\documentclass[a4paper,12pt,reqno]{amsart}
\usepackage{amsfonts, color}
\usepackage{amsmath}
\usepackage{amssymb}
\usepackage[a4paper]{geometry}
\usepackage{mathrsfs}
\usepackage[colorlinks]{hyperref}
\usepackage{hyperref}
\renewcommand\eqref[1]{(\ref{#1})}

\makeatletter
\newcommand*{\mint}[1]{%
  \mint@l{#1}{}%
}
\newcommand*{\mint@l}[2]{%
  \@ifnextchar\limits{%
    \mint@l{#1}%
  }{%
    \@ifnextchar\nolimits{%
      \mint@l{#1}%
    }{%
      \@ifnextchar\displaylimits{%
        \mint@l{#1}%
      }{%
        \mint@s{#2}{#1}%
      }%
    }%
  }%
}
\newcommand*{\mint@s}[2]{%
  \@ifnextchar_{%
    \mint@sub{#1}{#2}%
  }{%
    \@ifnextchar^{%
      \mint@sup{#1}{#2}%
    }{%
      \mint@{#1}{#2}{}{}%
    }%
  }%
}
\def\mint@sub#1#2_#3{%
  \@ifnextchar^{%
    \mint@sub@sup{#1}{#2}{#3}%
  }{%
    \mint@{#1}{#2}{#3}{}%
  }%
}
\def\mint@sup#1#2^#3{%
  \@ifnextchar_{%
    \mint@sup@sub{#1}{#2}{#3}%
  }{%
    \mint@{#1}{#2}{}{#3}%
  }%
}
\def\mint@sub@sup#1#2#3^#4{%
  \mint@{#1}{#2}{#3}{#4}%
}
\def\mint@sup@sub#1#2#3_#4{%
  \mint@{#1}{#2}{#4}{#3}%
}
\newcommand*{\mint@}[4]{%
  \mathop{}%
  \mkern-\thinmuskip
  \mathchoice{%
    \mint@@{#1}{#2}{#3}{#4}%
        \displaystyle\textstyle\scriptstyle
  }{%
    \mint@@{#1}{#2}{#3}{#4}%
        \textstyle\scriptstyle\scriptstyle
  }{%
    \mint@@{#1}{#2}{#3}{#4}%
        \scriptstyle\scriptscriptstyle\scriptscriptstyle
  }{%
    \mint@@{#1}{#2}{#3}{#4}%
        \scriptscriptstyle\scriptscriptstyle\scriptscriptstyle
  }%
  \mkern-\thinmuskip
  \int#1%
  \ifx\\#3\\\else_{#3}\fi
  \ifx\\#4\\\else^{#4}\fi
}
\newcommand*{\mint@@}[7]{%
  \begingroup
    \sbox0{$#5\int\m@th$}%
    \sbox2{$#5\int_{}\m@th$}%
    \dimen2=\wd0 %
    \let\mint@limits=#1\relax
    \ifx\mint@limits\relax
      \sbox4{$#5\int_{\kern1sp}^{\kern1sp}\m@th$}%
      \ifdim\wd4>\wd2 %
        \let\mint@limits=\nolimits
      \else
        \let\mint@limits=\limits
      \fi
    \fi
    \ifx\mint@limits\displaylimits
      \ifx#5\displaystyle
        \let\mint@limits=\limits
      \fi
    \fi
    \ifx\mint@limits\limits
      \sbox0{$#7#3\m@th$}%
      \sbox2{$#7#4\m@th$}%
      \ifdim\wd0>\dimen2 %
        \dimen2=\wd0 %
      \fi
      \ifdim\wd2>\dimen2 %
        \dimen2=\wd2 %
      \fi
    \fi
    \rlap{%
      $#5%
        \vcenter{%
          \hbox to\dimen2{%
            \hss
            $#6{#2}\m@th$%
            \hss
          }%
        }%
      $%
    }%
  \endgroup
}
\numberwithin{equation}{section}
\theoremstyle{plain}
\newtheorem{thm}{Theorem}[section]

\newtheorem{cor}[thm]{Corollary}
\newtheorem{lem}[thm]{Lemma}

\theoremstyle{definition}

\newtheorem{rem}[thm]{Remark}

\newcommand{\G}{\mathbb{G}}
\newcommand{\Ra}{\mathcal{R}^{\frac{a}{\nu}}}
\newcommand{\R}{\mathcal{R}}
\newcommand{\X}{\mathbb{X}}
\newcommand{\p}{p^{*}_{\beta}}

\title[Logarithmic Hardy-Rellich inequalities on Lie groups]{Logarithmic Hardy-Rellich inequalities on Lie groups}
\author[M. Chatzakou]{Marianna Chatzakou}
\address{
	Marianna Chatzakou:
	\endgraf
    Department of Mathematics: Analysis, Logic and Discrete Mathematics
    \endgraf
    Ghent University, Belgium
  	\endgraf
	{\it E-mail address} {\rm marianna.chatzakou@ugent.be}
		}
	
\author[A. Kassymov]{Aidyn Kassymov}
\address{
  Aidyn Kassymov:
  \endgraf
   \endgraf
  Department of Mathematics: Analysis, Logic and Discrete Mathematics
  \endgraf
  Ghent University, Belgium
  \endgraf
  and
  \endgraf
  Institute of Mathematics and Mathematical Modeling
  \endgraf
  125 Pushkin str.
  \endgraf
  050010 Almaty
  \endgraf
  Kazakhstan
  \endgraf
  and
  \endgraf
  Al-Farabi Kazakh National University
  \endgraf
   71 Al-Farabi avenue
   \endgraf
   050040 Almaty
   \endgraf
   Kazakhstan
  \endgraf
	{\it E-mail address} {\rm aidyn.kassymov@ugent.be} and {\rm kassymov@math.kz}}

\author[M. Ruzhansky]{Michael Ruzhansky}
\address{
  Michael Ruzhansky:
  \endgraf
  Department of Mathematics: Analysis, Logic and Discrete Mathematics
  \endgraf
  Ghent University, Belgium
  \endgraf
 and
  \endgraf
  School of Mathematical Sciences
  \endgraf
  Queen Mary University of London
  \endgraf
  United Kingdom
  \endgraf
  {\it E-mail address} {\rm michael.ruzhansky@ugent.be}
  }

\begin{document}

\thanks{The authors are supported by the FWO Odysseus 1 grant G.0H94.18N: Analysis and Partial Differential Equations and by the Methusalem programme of the Ghent University Special Research Fund (BOF) (Grant number 01M01021). Michael Ruzhansky is also supported by EPSRC grant EP/R003025/2. }

\begin{abstract} 
In this paper we obtain logarithmic Hardy and Rellich inequalities on general Lie groups. In the case of graded groups, we also show their refinements using the homogeneous Sobolev norms. In fact, we derive a family of weighted logarithmic Hardy-Rellich inequalities, for which logarithmic Hardy and Rellich inequalities are special cases. As a consequence of these inequalities, we also derive a Gross type logarithmic Hardy inequality on general stratified groups. An interesting feature of such estimate is that we consider the measure which is Gaussian only on the first stratum of the group. Such choice of the measure is natural in view of the known Gross type logarithmic Sobolev inequalities on stratified groups. The obtained results are new already in the setting of the Euclidean space $\mathbb R^n.$ Finally, we also present a simple argument for getting a logarithmic Poincar\'e inequality, as well as the logarithmic Hardy inequality for the fractional $p$-sub-Laplacian on homogeneous groups.
\end{abstract}

\maketitle

\tableofcontents

\section{Introduction}

In this paper we continue the investigation of the logarithmic functional inequalities in the setting of Lie groups. In \cite{CKR21b}, we have obtained the logarithmic Sobolev, Gagliardo-Nirenberg and Caffarelli-Kohn-Nirenberg inequalities on general Lie groups, and in this paper we deal with what can be called {\em (weighted) logarithmic Hardy-Rellich inequalities}, building on the terminology proposed by Del Pino, Dolbeault, Filippas and Tertikas in \cite{DDFT10}. The scope of results and the method of proof are rather different, meriting an independent presentation.

We will be also working in the settings of homogeneous, graded and stratified Lie groups where we can obtain some refinements or counterparts of the general results. 

The Sobolev inequality on the Euclidean space $\mathbb{R}^{n}$ takes the form
\begin{equation}
    \|u\|_{L^{p^{*}}(\mathbb{R}^{n})}\leq C \|\nabla u\|_{L^{p}(\mathbb{R}^{n})},
\end{equation}
where $1<p<n$, $p^{*}=\frac{np}{n-p},$ and $C=C(n,p)$ is a positive constant. Consequently, one also has the logarithmic Sobolev inequality  
\begin{equation}\label{ddlogin}
\int_{\mathbb{R}^{n}}\frac{|u|^{p}}{\|u\|^{p}_{L^{p}(\mathbb{R}^{n})}}\log\left(\frac{|u|^{p}}{\|u\|^{p}_{L^{p}(\mathbb{R}^{n})}}\right)dx\leq\frac{n}{p}\log\left(C\frac{\|\nabla u\|^{p}_{L^{p}(\mathbb{R}^{n})}}{\|u\|^{p}_{L^{p}(\mathbb{R}^{n})}}\right).
\end{equation}
It has a long history, starting with \cite{Gro75} and \cite{Wei78} for $p=2$, see  \cite{DD03} for some history for the range 
$1\leq  p<\infty$, as well as some best constants.

On the other hand, the classical Hardy inequality on $\mathbb R^n$ for $n\geq 3$ takes the form 
\begin{equation}\label{EQ:Hardy}
    \int_{\mathbb{R}^{n}}\frac{|u|^2}{|x|^2}dx\leq \frac{4}{(n-2)^2}\int_{\mathbb{R}^{n}}|\nabla u|^2 dx.
\end{equation}
It has a very long history of over 100 years, and we only refer to the open access book \cite{RS19} for a partial historical review. The logarithmic version of \eqref{EQ:Hardy}, called the {\em logarithmic Hardy inequality}, was proposed in \cite{DDFT10}, taking the form
\begin{equation}\label{EQ:Hardy-log}
    \int_{\mathbb{R}^{n}}\frac{|u|^2}{|x|^2}
    \log\left(|x|^{n-2}|u|^2\right)
    dx\leq \frac{n}{2}\log\left(C\int_{\mathbb{R}^{n}}|\nabla u|^2 dx\right),
\end{equation}
for all $u$ such that $\int_{\mathbb{R}^{n}}\frac{|u|^2}{|x|^2}dx=1.$ The $L^p$-Hardy inequality is an extension of \eqref{EQ:Hardy} for $1<p<n$, which can be stated as
\begin{equation}\label{EQ:Hardy-p}
    \int_{\mathbb{R}^{n}}\frac{|u|^p}{|x|^p}dx\leq \left(\frac{p}{n-p}\right)^p\int_{\mathbb{R}^{n}}|\nabla u|^p dx.
\end{equation}
We can also recall the Rellich inequality, for $n\geq 5$,
\begin{equation}\label{EQ:Rellich-2}
    \int_{\mathbb{R}^{n}}\frac{|u|^2}{|x|^4}dx\leq
    \frac{4}{n(n-4)}\int_{\mathbb{R}^{n}}|\Delta u|^2 dx.
\end{equation}
As usual, inequalities \eqref{EQ:Hardy} and \eqref{EQ:Rellich-2} hold for smooth compactly supported functions first, and then on spaces with respect to their completion in the corresponding norms on the right-hand sides of the respective inequalities.
The inequality \eqref{EQ:Rellich-2} goes back to Rellich \cite{Rel56}, and again, we refer to the open access book \cite{RS19} for many versions of this inequality and its history.

One of the aims of this paper is to obtain a version of a logarithmic $L^p$-Hardy inequality, that is, an analogue of \eqref{EQ:Hardy-log} for the inequality \eqref{EQ:Hardy-p}. We also want to obtain a logarithmic version of the Rellich inequality \eqref{EQ:Rellich-2}.
In fact, as a special case of the result in Theorem \ref{THM:weHargr} on $\mathbb R^n$, for all $1<p<\infty$ and $0<a<\frac{n}{p}$, we 
will obtain the family of {\em logarithmic Hardy-Rellich inequalities}
   \begin{equation}\label{weloghar1-Rn}
         \int_{\mathbb{R}^{n}}\frac{|u(x)|^{p}}{|x|^{ap}}\log\left(|x|^{n-ap}|u|^{p}\right) dx \leq \frac{n}{ap}\log\left(C\|(-\Delta)^{\frac{a}{2}}u\|^{p}_{L^{p}(\mathbb{R}^{n})}\right),
     \end{equation}
for all $u$ such that $\int_{\mathbb{R}^{n}}\frac{|u|^{p}}{|x|^{ap}}dx=1$. In particular, for $p=2$ and $a=1$, 
since $\|(-\Delta)^{\frac{1}{2}}u\|_{L^{2}}=\|\nabla u\|_{L^{2}}$,
this gives exactly the logarithmic Hardy inequality \eqref{EQ:Hardy-log}. For $a=2$, $1<p<\infty,$ and $2p<n$, this gives what we can call the {\em logarithmic Rellich inequality}
   \begin{equation}\label{weloghar1-Rn-Rellich}
         \int_{\mathbb{R}^{n}}\frac{|u(x)|^{p}}{|x|^{2p}}\log\left(|x|^{n-2p}|u|^{p}\right) dx \leq \frac{n}{2p}\log\left(C\int_{\mathbb{R}^{n}}|\Delta u|^{p} dx\right),
     \end{equation}
for all $u$ such that $\int_{\mathbb{R}^{n}}\frac{|u|^{p}}{|x|^{2p}}dx=1$.
In fact, we will also obtain the {\em weighted logarithmic Hardy-Rellich inequalities}: for all
 $1<p<\infty$ and $0\leq \beta<ap<n,$ we have
   \begin{equation}\label{weloghar1-Rn-w}
         \int_{\mathbb{R}^{n}}\frac{|u(x)|^{p}}{|x|^{ap-\beta}}\log\left(|x|^{(n-ap)(1-\frac{\beta}{ap-\beta})}|u|^{p}\right) dx \leq \frac{n-\beta}{ap-\beta}\log\left(C\|(-\Delta)^{\frac{a}{2}}u\|^{p}_{L^{p}({\mathbb{R}^{n}})}\right),
     \end{equation}
for all $u$ such that $\int_{\mathbb{R}^{n}}\frac{|u|^{p}}{|x|^{ap-\beta}}dx=1.$ These inequalities will be actually obtained in the setting of general Lie groups, but we specify them here for the Euclidean space $\mathbb R^n$ since they appear to be new already in the Euclidean setting. 

We can also mention a related family of {\em weighted logarithmic Sobolev inequalities} obtained in \cite[Theorem 4.1]{CKR21b}:
for all $1<p<\infty$ and $0\leq \beta <a p<n$ we have
\begin{equation}
    \label{gaus.log.sob-w-Rn-i}
    \int_{\mathbb R^n}|x|^{-\frac{\beta(n-ap)}{n-\beta}}|u|^p\log\left(|x|^{-\frac{\beta(n-ap)}{(n-\beta)}}|u|^p\right)\,dx \leq \frac{n-\beta}{ap-\beta}\log\left(C\|(-\Delta)^{\frac{a}{2}}u\|^{p}_{L^{p}({\mathbb{R}^{n}})}\right),\quad 
\end{equation}
for all $u$ such that $\int_{\mathbb R^n}|x|^{-\frac{\beta(n-ap)}{n-\beta}}|u|^p dx=1.$
The logarithmic Hardy inequality \eqref{weloghar1-Rn} can be seen as the critical case of the weighted logarithmic Sobolev inequalities \eqref{gaus.log.sob-w-Rn-i} for the critical order of the weight $\beta=ap.$

Another very important form of the logarithmic Sobolev inequalities is due to L. Gross \cite{Gro75}, who has shown the logarithmic Sobolev inequality in 
the form: 
\begin{equation}\label{EQ:Gross}
    \int_{\mathbb{R}^{n}}|g(x)|^{2}\log  \left(|g(x)|\right)d\mu(x)
    \leq  \int_{\mathbb{R}^{n}}|\nabla g(x)|^{2}d\mu(x),
\end{equation}
with the Gaussian measure $d\mu(x)=(2\pi)^{-n/2}e^{-|x|^2/2}dx$, for all 
$u$ such that $\|g\|_{L^{2}(\mu)}=1$. This inequality allows passing to infinite dimensions, and has many variants and applications; for a small selection of papers see e.g.
\cite{Ad79, AC79, Ros76, Wei78, Tos97, SZ92}. We also refer to \cite{Dav89} for an abstract discussion of log-Sobolev inequalities for generators of symmetric Markov semigroups.

In \cite{CKR21b}, we have proved a family of weighted Gross type logarithmic Sobolev inequalities, allowing also for different choices of weights. For example, consider the family of Euclidean norms on $\mathbb R^n$ given by
$|x|_r=(|x_1|^r+\cdots+|x_n|^r)^{\frac{1}{r}}$ for $1\leq r<\infty$, and by $|x|_\infty=\max\limits_{1\leq j\leq n}|x_j|$, for $x\in\mathbb R^n.$ 
Let $\mu$ be the Gaussian measure on $\mathbb{R}^{n}$ given by 
$$d\mu=\gamma_r e^{-\frac{|x|_2^2}{2}}dx, \quad 1\leq r\leq\infty,$$ 
with some (specified) normalisation constant $\gamma_r$. 
Then in \cite{CKR21b}, we have obtained the following {\em weighted Gross type logarithmic Sobolev inequalities}:
\begin{equation}
    \label{gaus.log.sob-w-Rn-i2}
    \int_{\mathbb R^n}|x|_r^{-\frac{\beta(n-2)}{n-\beta}}|g|^2\log\left(|x|_r^{-\frac{\beta(n-2)}{2(n-\beta)}}|g|\right)\,d\mu \leq \int_{\mathbb R^n} |\nabla g|^2\,d\mu,\quad 
    0\leq \beta <2<n,
\end{equation}
for all $g$ such that $\||x|_r^{-\frac{\beta(n-2)}{2(n-\beta)}}g\|_{L^2(\mu)}=1.$
We note that for $\beta=0$ this implies the Gross inequality \eqref{EQ:Gross}. In this paper, as a consequence of the logarithmic Hardy inequality, we obtain the critical case of \eqref{gaus.log.sob-w-Rn-i2} with $\beta=2$, which we can probably call the {\em Gross type logarithmic Hardy inequality}
\begin{equation}
    \label{gaus.log.Hard-i}
    \int_{\mathbb R^n}\frac{|g(x)|^2}{|x|_r^2} \log \left(|x|_r^{\frac{n-2}{2}}|g(x)| \right)\,d\mu \leq \int_{\mathbb R^n} |\nabla g|^2\,d\mu\,,
\end{equation}
for all $g$ such that $\left\|\frac{g}{|x|_r} \right\|_{L^2(\mu)}=1.$
By taking $\beta=1$ in Theorem \ref{thm.gaus.Hardy}, we also get the following interesting inequality:
 \begin{equation}
    \label{gaus.log.Hard-Rn-b1}
\int_{\mathbb R^n}\frac{|g(x)|^2}{|x|_r} \log \left(|g(x)| \right)\,d\mu \leq \int_{\mathbb R^n}|\nabla g|^2\,d\mu\,,
\end{equation}
for all $g$ such that $\left\|\frac{g}{|x|_r^{1/2}} \right\|_{L^2(\mu)}=1$.

 Similarly to \cite{CKR21b} for (weighted) logarithmic Sobolev inequalities, here we prove a version of \eqref{gaus.log.Hard-i} on stratified Lie groups. In this case, if $\G$ is a stratified Lie group of homogeneous dimension $Q$, the inequality \eqref{gaus.log.Hard-i} takes the form 
 \begin{equation}
    \label{gaus.log.Hard-G}
    \int_{\G}\frac{|g(x)|^2}{|x|^2} \log \left(|x|^{\frac{Q-2}{2}}|g(x)| \right)\,d\mu \leq \int_{\G} |\nabla_H g|^2\,d\mu\,,
\end{equation}
for all $g$ such that $\left\|\frac{g}{|x|} \right\|_{L^2(\mu)}=1.$ Here $\nabla_H$ is the horizontal gradient and $|\cdot|$ is any homogeneous quasi-norm on $\G$, while $\mu=\mu_1\otimes\mu_2$, where $\mu_1$ is a Gaussian measure on the first stratum of $\G$ and $\mu_2$ is the Lebesgue measure on the other strata; we refer to Theorem \ref{thm.gaus.Hardy} for the precise formulation of the weighted version of \eqref{gaus.log.Hard-G}. In particular, by taking $\beta=1$ in Theorem \ref{thm.gaus.Hardy}, we get the following interesting inequality:
 \begin{equation}
    \label{gaus.log.Hard-G-b1}
\int_{\G}\frac{|g(x)|^2}{|x|} \log \left(|g(x)| \right)\,d\mu \leq \int_{\G}|\nabla_{H}g|^2\,d\mu\,,
\end{equation}
for all $g$ such that $\left\|\frac{g}{|x|^{1/2}} \right\|_{L^2(\mu)}=1$.
 
From the point of view of the analysis on groups, the subject of inequalities has been extensively investigated. We can refer to 
\cite{VSCC93} for Sobolev inequalities on general unimodular Lie groups, \cite{AR20} for Sobolev inequalities on general locally compact unimodular groups, to \cite{BPTV19,BPV21} for Sobolev inequalities on general noncompact Lie groups, to \cite{FR17, FR16, RTY20} for graded groups, 
and to \cite{RY18a} for Hardy-Sobolev inequalities on general Lie groups.
For versions of the Euclidean weighted logarithmic Sobolev inequalities (with one weight) and Hardy inequalities with special weights we refer to 
\cite{Das21}. Hardy, Rellich and other inequalities with weights have been intensively investigated on homogeneous groups of different types, see \cite{RS19} and references therein. For logarithmic inequalities of Gross type, as the subject is immense, we only refer to surveys \cite{AB00, GZ03} as well as to some recent works on related coercive inequalities on stratified/Carnot groups e.g. \cite{BZ21a, CFZ21, HZ09}.

In this paper we establish a family of weighted logarithmic Hardy-Rellich inequalities on general Lie groups. The only assumption that we impose is that the group is connected, {\em otherwise it can be compact or noncompact, it can be unimodular or non-unimodular, and it can have polynomial or exponential volume growth at infinity.} Let us give an example.
Let $\G$ be a connected Lie group with local dimension $d$ and let $1<p<\infty$, $a>0$, be such that $0\leq \beta<ap<Q.$ Then we have 
    \begin{equation}\label{weloghargennon-i}
         \int_{\mathbb{G}}{\frac{|u(x)|^{p}}{|x|_{CC}^{ap-\beta}}}\log\left({|x|_{CC}^{(d-ap)(1-\frac{\beta}{ap-\beta})}|u|^{p}}\right) d\lambda(x) \leq \frac{d-\beta}{ap-\beta}\log\left(C{\|u\|^{p}_{L^{p}_{a}(\lambda)}}\right),
     \end{equation}
for all $u$ such that $\left\|\frac{u}{|x|_{CC}^{\frac{ap-\beta}{p}}}\right\|_{L^{p}(\lambda)}=1,$ where $d\lambda$ is the left Haar measure on $\G$, and $|x|_{CC}$ is the Carnot-Carath\'eodory distance from $x$ to the identity element $e$. In fact, the inequality \eqref{weloghargennon-i} also holds for the Carnot-Carath\'eodory distance replaced by the Riemannian distance on $\G$, compatible with the group structure. 
   As a special case of \eqref{weloghargennon-i} for $\beta=0$ we have 
   \begin{equation}\label{weloghar1non-i}
         \int_{\mathbb{G}}\frac{|u(x)|^{p}}{|x|_{CC}^{ap}}\log\left(|x|_{CC}^{d-ap}|u|^{p}\right) d\lambda(x) \leq \frac{d}{ap}\log\left(C\|u\|^{p}_{L_{a}^{p}(\lambda)}\right),
     \end{equation}
 for all $u$ such that $\left\|\frac{u}{|x|_{CC}^{a}}\right\|_{L^{p}(\lambda)}=1.$ 
 Thus for $a=1$ we obtain the logarithmic Hardy inequality, and for $a=2$ the logarithmic Rellich inequality. In fact, in
     Theorem \ref{THM:weHarnon} we also give these inequalities for all $u\not=0$, without the normalisation condition 
     $\left\|\frac{u}{|x|_{CC}^{\frac{ap-\beta}{p}}}\right\|_{L^{p}(\lambda)}=1.$
 
 Consequently, we also give several related versions of these inequalities as well as their refinements for some classes of Lie groups. For example, for graded groups we obtain \eqref{weloghargennon-i} with the homogeneous Sobolev space on the right and arbitrary quasi-norm. More precisely, let
 $\G$ be a graded Lie group of homogeneous dimension $Q$ and let
$\R$ be a positive Rockland operator of homogeneous degree $\nu$. Let $|\cdot|$ be an arbitrary homogeneous quasi-norm. Then for all $1<p<\infty$ and $0\leq \beta<ap<Q$ we have the {\em weighted logarithmic Hardy-Rellich inequalities}
   \begin{equation}\label{weloghar1-i}
         \int_{\mathbb{G}}\frac{|u(x)|^{p}}{|x|^{ap-\beta}}\log\left(|x|^{(Q-ap)(1-\frac{\beta}{ap-\beta})}|u|^{p}\right) dx \leq \frac{Q-\beta}{ap-\beta}\log\left(C\|\R^{\frac{a}{\nu}}u\|^{p}_{L^{p}(\G)}\right),
     \end{equation}
    for all $u$ such that  $\int_{\G}\frac{|u|^{p}}{|x|^{ap-\beta}}dx=1.$
 In Theorem \ref{THM:weHargr} we also give a version of this inequality without the normalisation condition  $\int_{\G}\frac{|u|^{p}}{|x|^{ap-\beta}}dx=1.$ 
 By taking $\beta=\frac{ap}{2}$, as a consequence of \eqref{weloghar1-i} we obtain an interesting inequality
   \begin{equation}\label{weloghar1non-nw1}
         \int_{\mathbb{G}}\frac{|u(x)|^{p}}{|x|^{\frac{ap}{2}}}\log\left(|u|\right) dx \leq \frac{2Q-ap}{ap^2}\log\left(C\|\R^{\frac{a}{\nu}}u\|^{p}_{L^{p}(\G)}\right),
     \end{equation}
    for all $u$ such that $\int_{\G}\frac{|u|^{p}}{|x|^{\frac{ap}{2}}}dx=1$.
 
 In particular, if $\G$ is a stratified group and $\nabla_H$ and $\Delta_\G$ are the horizontal gradient and the sub-Laplacian on $\G$, respectively, then for $p=2$, $\beta=0$, and $a=1$, we get the horizontal logarithmic $L^2$-Hardy inequality
    \begin{equation}\label{weloghar1-i-h}
         \int_{\mathbb{G}}\frac{|u(x)|^{2}}{|x|^{2}}\log\left(|x|^{Q-2}|u|^{2}\right) dx \leq \frac{Q}{2}\log\left(C\int_\G|\nabla_H u|^{2} dx\right),\quad Q\geq 3,
     \end{equation}
 and for $p=2$, $\beta=0$, and $a=2$, we get the logarithmic $L^2$-Rellich inequality
     \begin{equation}\label{weloghar1-i-r}
         \int_{\mathbb{G}}\frac{|u(x)|^{2}}{|x|^{4}}\log\left(|x|^{Q-4}|u|^{2}\right) dx \leq \frac{Q}{4}\log\left(C\int_\G|\Delta_\G u|^{2} dx\right),\quad Q\geq 5.
     \end{equation}
  Inequality \eqref{weloghar1-i} can be compared with the logarithmic Sobolev inequality on graded Lie groups established in \cite{CKR21b}, in a particular (unweighted) case, for
  $1<p<\infty$ and $0<a<\frac{Q}{p}$, given by 
\begin{equation}\label{a2=0i}
\int_{\mathbb{G}}\frac{|u|^{p}}{\|u\|^{p}_{L^{p}(\mathbb{G})}}\log\left(\frac{|u|^{p}}{\|u\|^{p}_{L^{p}(\mathbb{G})}}\right)dx\leq \frac{Q}{ap}\log\left(A\frac{\|\R^\frac{a}{\nu}u\|^p_{{L}^p(\G)}}{\|u\|^{p}_{L^{p}(\G)}}\right),
\end{equation}
for any Rockland operator $\R$ of homogeneous degree $\nu$. For the case $p=1$ in the form of the Shannon inequality on general homogeneous groups we refer to \cite{CKR21a}.

In Theorem \ref{REM:fractional} we also give the logarithmic Hardy-Rellich inequalities for fractional $p$-sub-Laplacians on general homogeneous groups. When working with the fractional Laplacians, we can avoid using the differential structure of the group, thus making an extension of \eqref{weloghar1-i} to general homogeneous groups possible. For the fractional Laplacians, 
the fractional logarithmic Sobolev inequality on the Heisenberg group and on homogeneous  groups was proved in \cite {FNQ18} and \cite{KRS20}, respectively.   

We will give the very brief definitions for the above objects in the subsequent sections. For different classes of homogeneous groups (graded, stratified) we can refer to a recent presentation in the open access book \cite{FR16}, based on the 
fundamental work of Folland and Stein \cite{FS}. 

\section{Logarithmic Hardy-Rellich inequalities on general Lie groups}
\label{SEC:generalLie}

In this section we prove the logarithmic Hardy-Rellich inequalities on general Lie groups. 
Let $\G$ be a noncompact connected Lie group with identity element $e$, and
let $X = \{X_1, . . . , X_n\}$ be a family of linearly independent, left-invariant vector fields on $\G$ satisfying H\"{o}rmander’s condition. 
Let $d\rho$ and $d\lambda$ be the right and left Haar measures on $\G$, respectively. If $\delta$ is the modular function on $\G$, we have $d\lambda = \delta d\rho$. 

To put the construction of the sub-Laplacian with drift into perspective, let $\chi$ be a continuous positive character of $\G$, and 
denote by $\mu_{\chi}$ the measure on $\G$ with density equal to $\chi$, with respect to the right Haar measure $\rho$ of $\G$, that is, $d\mu_{\chi} = \chi d\rho.$ Then it was shown in \cite{HMM05} that the  differential operator
\begin{equation}
    \Delta_{\chi}=-\sum_{i=1}^{n}\left(X_{j}^{2}+c_{j}X_{j}\right),
    \; c_{j} = (X_{j}\chi)(e), \; j = 1, \ldots , n,
\end{equation}
with domain $C_{0}^{\infty}(\G)$ is essentially self-adjoint on $L^{2}(\mu_{\chi})$.
Then one can define the Sobolev spaces on $\G$, associated to $\chi$, by
\begin{equation}
    L^{p}_{a}(\mu_{\chi}):=\{u:u\in L^{p}(\mu_{\chi}),\Delta^{\frac{a}{2}}_{\chi}u\in L^{p}(\mu_{\chi}) \},
\end{equation}
with the norm
\begin{equation}
    \|u\|_{L^{p}_{a}(\mu_{\chi})}:=\|u\|_{L^{p}(\mu_{\chi})}+\|\Delta^{\frac{a}{2}}_{\chi}u\|_{L^{p}(\mu_{\chi})}.
\end{equation}
In particular, for the left Haar measure $d\lambda$ we have, with $\chi=\delta$,
\begin{equation}
    L^{p}_{a}(\lambda):=\{u:u\in L^{p}(\lambda),\Delta^{\frac{a}{2}}_{\delta }u\in L^{p}(\lambda) \},
\end{equation}
with the norm
\begin{equation*}
    \|u\|_{L^{p}_{a}(\lambda)}:=\|u\|_{L^{p}(\lambda)}+\|\Delta^{\frac{a}{2}}_{\delta}u\|_{L^{p}(\lambda)}.
\end{equation*}
The embedding properties of the Sobolev spaces $L^{p}_{a}(\mu_{\chi})$ were established in \cite{BPTV19} for noncompact Lie groups and in \cite{RY18a} for general Lie groups by different methods. We will use the latter paper because this also gives a family of weighted Hardy-Sobolev inequalities important for us here, and also includes the case of compact Lie groups.

The family $X = \{X_1, . . . , X_n\}$ induces the Carnot-Carath\'{e}odory distance $d_{CC}(\cdot, \cdot)$ on $\G$. Let $B(c_{B},r_{B})$ denote the ball of radius $r_{B}$ with respect to this distance, centred at $c_{B}$. We will denote by 
\begin{equation}\label{EQ:dcc}
|x|_{CC}=d_{CC}(e,x)
\end{equation}
the Carnot-Carath\'{e}odory distance between $x$ and and the identity element $e$ of $\G$. If we denote by $V(r):=\rho(B_{r})$ the volume of the ball $B_{r}:=B(e,r)$ with respect to the right Haar measure on $\G$, then it is well-known that there exist two constants $d=d(\G,X)$ and $D=D(\G)$ such that
\begin{equation*}
    V(r)\approx r^{d},\,\,\,\,\forall r\in (0,1],
\end{equation*}
and
\begin{equation*}
    V(r)\lesssim e^{Dr},\,\,\,\,\forall r\in (1,\infty).
\end{equation*}
In this case $d$ and $D$ are called the local and global dimensions of the metric measure space $(\G,d_{CC}, \rho)$, respectively. The volume growth at infinity can be also polynomial.

In the case when $\G$ is a compact Lie group, we have that $r$ is bounded, and the dimension $d$ coincides with the Hausdorff dimension of $\G$ with respect to the family $X = \{X_1, . . . , X_n\}$, see e.g. \cite{AR20}.
Since $\G$ is then also unimodular, we have that the modular function is $\delta\equiv 1$, and it means that $d\lambda=d\rho$. Furthermore, with the character $\chi=1$, the Sobolev space $L^{p}_{a}(\mu_{\chi})$ coincides with Sobolev space defined by the sub-Laplacian $\Delta_\G=\sum\limits_{i=1}^{n}X_{i}^{2}$. 

Let us now recall the Hardy-Sobolev inequality from \cite{RY18a}, for general connected Lie groups. 

\begin{thm}\label{Harsobinnthm}
Let $\G$ be a connected Lie group. Let $0\leq \beta < d$ and $1< p,q <\infty$.
\begin{itemize}
    \item[(a)] If $a< \frac{d}{p}$, then we have 
  \begin{equation}\label{HSinnon}
      \left\|\frac{u}{|x|_{CC}^{\frac{\beta}{q}}}\right\|_{L^{q}(\lambda)}\leq C \|u\|_{L^{p}_{a}(\lambda)},
  \end{equation}
  for all $q\geq p$ such that $\frac{1}{p}-\frac{1}{q}\leq \frac{a}{d}-\frac{\beta}{dq}$;
  \item[(b)] If $\frac{d}{p}\leq a$, then \eqref{HSinnon} holds for all $q\geq p$.
\end{itemize}
\end{thm}
 We can now present the weighted logarithmic Hardy-Rellich inequality for general Lie groups (compact and non-compact).
 
\begin{thm}\label{THM:weHarnon}
Let $\G$ be a connected Lie group with local dimension $d$ and let $1<p<\infty$, $a>0$, be such that $0\leq \beta<ap<d.$ Then we have 
    \begin{equation}\label{weloghargennon}
         \int_{\mathbb{G}}\frac{\frac{|u(x)|^{p}}{|x|_{CC}^{ap-\beta}}}{\left\|\frac{u}{|x|_{CC}^{\frac{ap-\beta}{p}}}\right\|^{p}_{L^{p}(\lambda)}}\log\left(\frac{|x|_{CC}^{(d-ap)(1-\frac{\beta}{ap-\beta})}|u|^{p}}{\left\|\frac{u}{|x|_{CC}^{\frac{ap-\beta}{p}}}\right\|^{p}_{L^{p}(\lambda)}}\right) d\lambda(x) \leq \frac{d-\beta}{ap-\beta}\log\left(C\frac{\|u\|^{p}_{L^{p}_{a}(\lambda)}}{\left\|\frac{u}{|x|_{CC}^{\frac{ap-\beta}{p}}}\right\|^{p}_{L^{p}(\lambda)}}\right),
     \end{equation}
     for all $u\in L_{a}^{p}(\lambda)\backslash\{0\}$ such that $|x|_{CC}^{-\frac{ap-\beta}{p}}|u|\in L^{p}(\lambda)$.
     
   In particular, for all $u$ such that $\int_{\G}\frac{|u|^{p}}{|x|_{CC}^{ap-\beta}}d\lambda(x)=1$, we have 
   \begin{equation}\label{weloghar1non}
         \int_{\mathbb{G}}\frac{|u(x)|^{p}}{|x|_{CC}^{ap-\beta}}\log\left(|x|_{CC}^{(d-ap)(1-\frac{\beta}{ap-\beta})}|u|^{p}\right) d\lambda(x) \leq \frac{d-\beta}{ap-\beta}\log\left(C\|u\|^{p}_{L_{a}^{p}(\lambda)}\right).
     \end{equation}
\end{thm}
\begin{proof}
 Let us choose $q\in(p,p_{\beta}^{*})$, where $\p=\frac{(d-\beta)p}{d-ap}$. Let  $\theta=\frac{(\p-q)p}{\p-p}$, then we have $\theta\in(0,p)$. Using this, we prepare some computations:
\begin{multline}\label{wevyrtheta}
   \theta=\frac{(\p-q)p}{\p-p}=\frac{\left(\frac{(d-\beta)p}{d-ap}-q\right)p}{\frac{(d-\beta)p}{d-ap}-p} \\ =\frac{\frac{1}{d-ap}\left((d-\beta)p-q(d-ap)\right)p}{\frac{1}{d-ap}((d-\beta)p-dp+ap^{2})}=\frac{(d-\beta)p}{ap-\beta }-\frac{(d-ap)q}{ap-\beta},
\end{multline}
\begin{equation}\label{wep(qth)}
    p\frac{q-\theta}{p-\theta}=p\frac{q-p\frac{\p-q}{\p-p}}{p-p\frac{\p-q}{\p-p}}=\frac{q(\p-p)-p(\p-q)}{\p-p-\p+q}=\frac{\p(q-p)}{q-p}=\p,
\end{equation}
\begin{equation}\label{weptp}
    \frac{p-\theta}{p}=1-\frac{\theta}{p}=\frac{q-p}{\p-p},
\end{equation}
and 
\begin{equation}\label{hbeta}
    -\frac{\beta \theta}{p}=\beta\frac{p-\theta}{p}-\beta.
\end{equation}
By using H\"{o}lder's inequality with $\frac{\theta}{p}+\frac{p-\theta}{p}=1$ and the above calculations we get,
\begin{equation*}
    \begin{split}
        \int_{\G}\frac{|u|^{q}}{|x|_{CC}^{d-\beta-\frac{q}{p}(d-ap)}}d\lambda(x) &=\int_{\G}\frac{|u|^{\theta}}{|x|_{CC}^{\theta a}}\frac{|u|^{q-\theta}}{|x|_{CC}^{d-\beta-\frac{q}{p}(d-ap)-\theta a}}d\lambda(x)\\&
        \stackrel{\eqref{wevyrtheta}}=\int_{\G}\frac{|u|^{\theta}}{|x|_{CC}^{\theta a}}\frac{|u|^{q-\theta}}{|x|_{CC}^{-\frac{\theta \beta}{p}}}d\lambda(x)\\&
        \stackrel{\eqref{hbeta}}= \int_{\G}\frac{|u|^{\theta}}{|x|_{CC}^{\theta a}}\frac{|u|^{q-\theta}}{|x|^{\beta\frac{p-\theta}{p}-\beta}}d\lambda(x)\\&
        =\int_{\G}\frac{|u|^{\theta}}{|x|_{CC}^{\theta a-\beta}}\frac{|u|^{q-\theta}}{|x|_{CC}^{\beta\frac{p-\theta}{p}}}d\lambda(x)\\&
        \leq \left(\int_{\G}\frac{|u|^{p}}{|x|_{CC}^{ap-\frac{\beta p}{\theta}}}d\lambda(x)\right)^{\frac{\theta}{p}}\left(\int_{\G}\frac{|u|^{\frac{p(q-\theta)}{p-\theta}}}{|x|_{CC}^{\beta}}d\lambda(x)\right)^{\frac{p-\theta}{p}}\\&
        \stackrel{\eqref{wep(qth)},\eqref{weptp}}=\left(\int_{\G}\frac{|u|^{p}}{|x|_{CC}^{ap-\beta\frac{\p-p}{\p-q}}}d\lambda(x)\right)^{\frac{\p-q}{\p-p}}\left(\int_{\G}\frac{|u|^{\p}}{|x|_{CC}^{\beta}}d\lambda(x)\right)^{\frac{q-p}{\p-p}}.
    \end{split}
\end{equation*}
Let us write $q=p+r$ with $r>0$. Then, the last inequality can be rewritten as follows:
\begin{multline}\label{wesrnon}
        \int_{\G}\frac{|u|^{r+p}}{|x|_{CC}^{d-\beta-\frac{(r+p)}{p}(d-ap)}}d\lambda(x)
        \\ \leq \left(\int_{\G}\frac{|u|^{p}}{|x|_{CC}^{ap-\beta\frac{\p-p}{\p-(r+p)}}}d\lambda(x)\right)^{\frac{\p-(r+p)}{\p-p}}\left(\int_{\G}\frac{|u|^{\p}}{|x|_{CC}^{\beta}}d\lambda(x)\right)^{\frac{(r+p)-p}{\p-p}}.
\end{multline}
Also, we have for $r=0$ that
\begin{equation*}\label{webrnon}
    \begin{split}
        \int_{\G}\frac{|u|^{p}}{|x|_{CC}^{d-\beta-\frac{p}{p}(d-ap)}}d\lambda(x)&
        =\left(\int_{\G}\frac{|u|^{p}}{|x|_{CC}^{ap-\beta\frac{\p-p}{\p-p}}}d\lambda(x)\right)^{\frac{\p-p}{\p-p}}\left(\int_{\G}\frac{|u|^{\p}}{|x|_{CC}^{\beta}}d\lambda(x)\right)^{\frac{p-p}{\p-p}}.
    \end{split}
\end{equation*}
By using this and \eqref{wesrnon},
and by taking the limit as $r\rightarrow 0$, we have
\begin{equation}\label{welogosn2non}
\begin{split}
        &\lim_{r\rightarrow 0}\frac{1}{r}\int_{\G}\left(\frac{|u|^{r+p}}{|x|_{CC}^{d-\beta-\frac{(r+p)}{p}(d-ap)}}-\frac{|u|^{p}}{|x|_{CC}^{d-\beta-\frac{p}{p}(d-ap)}}\right)d\lambda(x)\\&
        \leq \lim_{r\rightarrow 0}\frac{1}{r}\Biggl{[}\left(\int_{\G}\frac{|u|^{p}}{|x|_{CC}^{ap-\beta\frac{\p-p}{\p-(r+p)}}}d\lambda(x)\right)^{\frac{\p-(r+p)}{\p-p}}\left(\int_{\G}\frac{|u|^{\p}}{|x|_{CC}^{\beta}}d\lambda(x)\right)^{\frac{(r+p)-p}{\p-p}}\\&
        -\left(\int_{\G}\frac{|u|^{p}}{|x|_{CC}^{ap-\beta\frac{\p-p}{\p-p}}}d\lambda(x)\right)^{\frac{\p-p}{\p-p}}\left(\int_{\G}\frac{|u|^{\p}}{|x|_{CC}^{\beta}}d\lambda(x)\right)^{\frac{p-p}{\p-p}}\Biggl{]}.
\end{split}
\end{equation}
Now, first let us compute the left hand side of the inequality \eqref{welogosn2non}. Using the l'H\^{o}pital rule in the variable $r$,  we have
\begin{equation}\label{otkrlevchnon}
    \begin{split}
        \lim_{r\rightarrow 0}\frac{1}{r}\left(\frac{|u|^{r+p}}{|x|_{CC}^{d-\beta-\frac{(r+p)}{p}(d-ap)}}-\frac{|u|^{p}}{|x|_{CC}^{d-\beta-\frac{p}{p}(d-ap)}}\right)&=\lim_{r\rightarrow 0}\frac{d}{dr}\frac{|u|^{r+p}}{|x|_{CC}^{d-\beta-\frac{(r+p)}{p}(d-ap)}}\\&
        =\frac{|u|^{p}\log(|u|^{p}|x|_{CC}^{d-ap})}{p|x|_{CC}^{ap-\beta}}.
    \end{split}
\end{equation}
Let us now compute the right hand side of  \eqref{welogosn2non}. Let us denote $z(r):=(f(r))^{g(r)}$, where
\begin{equation}
    f(r)=\int_{\G}\frac{|u|^{p}}{|x|_{CC}^{ap-\beta\frac{\p-p}{\p-(r+p)}}}d\lambda(x),
\end{equation}
and 
\begin{equation}
    g(r)=\frac{\p-(r+p)}{\p-p}.
\end{equation}
By taking derivative, we get
\begin{equation}\label{derfrnon}
    \begin{split}
       \frac{df(r)}{dr}&=\frac{d}{dr}\int_{\G}\frac{|u|^{p}}{|x|_{CC}^{ap-\beta\frac{\p-p}{\p-(r+p)}}}d\lambda(x)\\&
       =\frac{\beta(\p-p)}{(\p-(r+p))^{2}}\int_{\G}\frac{|u|^{p}}{|x|_{CC}^{ap-\beta\frac{\p-p}{\p-(r+p)}}}\log(|x|_{CC})d\lambda(x),
    \end{split}
\end{equation}
and 
\begin{equation}\label{dergrnon}
    \frac{dg(r)}{dr}=\frac{d}{dr}\frac{\p-(r+p)}{\p-p}=-\frac{1}{\p-p}.
\end{equation}
Then we have that the derivative of $z(r)$ is given by 
\begin{equation}\label{compzrnon}
\begin{split}
    \frac{dz(r)}{dr}&=(f(r))^{g(r)}\left(\frac{dg(r)}{dr}\log(f(r))+\frac{g(r)\frac{df(r)}{dr}}{f(r)}\right)\\&
    \stackrel{\eqref{derfrnon},\eqref{dergrnon}}=z(r)\Biggl{(}-\frac{1}{\p-p}\log\left(\int_{\G}\frac{|u|^{p}}{|x|_{CC}^{ap-\beta\frac{\p-p}{\p-(r+p)}}}d\lambda(x)\right)\\&
    +\frac{\beta}{\p-(r+p)}\frac{\int_{\G}\frac{|u|^{p}}{|x|_{CC}^{ap-\beta\frac{\p-p}{\p-(r+p)}}}\log(|x|_{CC})d\lambda(x)}{\int_{\G}\frac{|u|^{p}}{|x|_{CC}^{ap-\beta\frac{\p-p}{\p-(r+p)}}}d\lambda(x)}\Biggl{)}.
\end{split}
\end{equation}
At the same time, we also have
\begin{equation}\label{compharsobnon}
    \begin{split}
      \frac{d}{dr}\left(\int_{\G}\frac{|u|^{\p}}{|x|_{CC}^{\beta}}d\lambda(x)\right)^{\frac{(r+p)-p}{\p-p}}&=  \frac{1}{\p-p}\left(\int_{\G}\frac{|u|^{\p}}{|x|_{CC}^{\beta}}d\lambda(x)\right)^{\frac{(r+p)-p}{\p-p}}\log\left(\int_{\G}\frac{|u|^{\p}}{|x|_{CC}^{\beta}}d\lambda(x)\right).
    \end{split}
\end{equation}
By using these last calculations, and applying l'H\^{o}pital's rule, we compute further:
\begin{equation}
    \begin{split}
        &\lim_{r\rightarrow 0}\frac{1}{r}\Biggl{[}\left(\int_{\G}\frac{|u|^{p}}{|x|_{CC}^{ap-\beta\frac{\p-p}{\p-(r+p)}}}d\lambda(x)\right)^{\frac{\p-(r+p)}{\p-p}}\left(\int_{\G}\frac{|u|^{\p}}{|x|_{CC}^{\beta}}d\lambda(x)\right)^{\frac{(r+p)-p}{\p-p}}\\&
        -\left(\int_{\G}\frac{|u|^{p}}{|x|_{CC}^{ap-\beta\frac{\p-p}{\p-p}}}d\lambda(x)\right)^{\frac{\p-p}{\p-p}}\left(\int_{\G}\frac{|u|^{\p}}{|x|_{CC}^{\beta}}d\lambda(x)\right)^{\frac{p-p}{\p-p}}\Biggl{]}\\&
        =\lim_{r\rightarrow 0}\frac{d}{dr}\left[\left(\int_{\G}\frac{|u|^{p}}{|x|_{CC}^{ap-\beta\frac{\p-p}{\p-(r+p)}}}d\lambda(x)\right)^{\frac{\p-(r+p)}{\p-p}}\left(\int_{\G}\frac{|u|^{\p}}{|x|_{CC}^{\beta}}d\lambda(x)\right)^{\frac{(r+p)-p}{\p-p}}\right]\\&
        =\lim_{r\rightarrow 0}\frac{d}{dr}\left(z(r)\left(\int_{\G}\frac{|u|^{\p}}{|x|_{CC}^{\beta}}d\lambda(x)\right)^{\frac{(r+p)-p}{\p-p}}\right)\\&
        =\lim_{r\rightarrow 0}\left[\frac{dz(r)}{dr}\left(\int_{\G}\frac{|u|^{\p}}{|x|_{CC}^{\beta}}d\lambda(x)\right)^{\frac{(r+p)-p}{\p-p}}+z(r)\frac{d}{dr}\left(\int_{\G}\frac{|u|^{\p}}{|x|_{CC}^{\beta}}d\lambda(x)\right)^{\frac{(r+p)-p}{\p-p}}\right]\\&
    \stackrel{\eqref{compzrnon},\eqref{compharsobnon}}=\frac{\int_{\G}\frac{|u|^{p}}{|x|_{CC}^{ap-\beta}}d\lambda(x)}{\p-p}\Biggl{[}-\log\left(\int_{\G}\frac{|u|^{p}}{|x|_{CC}^{ap-\beta}}d\lambda(x)\right)\\&
    +\beta\frac{\int_{\G}\frac{|u|^{p}}{|x|_{CC}^{ap-\beta}}\log(|x|_{CC})d\lambda(x)}{\int_{\G}\frac{|u|^{p}}{|x|_{CC}^{ap-\beta}}d\lambda(x)}+\log\left(\int_{\G}\frac{|u|^{\p}}{|x|_{CC}^{\beta}}d\lambda(x)\right)\Biggl{]}\\&
    =\frac{\int_{\G}\frac{|u|^{p}}{|x|_{CC}^{ap-\beta}}d\lambda(x)}{\p-p}\Biggl{[}\beta\frac{\int_{\G}\frac{|u|^{p}}{|x|_{CC}^{ap-\beta}}\log(|x|_{CC})d\lambda(x)}{\int_{\G}\frac{|u|^{p}}{|x|_{CC}^{ap-\beta}}d\lambda(x)}+\log\frac{\left(\int_{\G}\frac{|u|^{\p}}{|x|_{CC}^{\beta}}d\lambda(x)\right)}{\left(\int_{\G}\frac{|u|^{p}}{|x|_{CC}^{ap-\beta}}d\lambda(x)\right)}\Biggl{]}\\&
    =\frac{I_{1}}{\p-p}\left(\beta\frac{I_{3}}{I_{1}}+\log\frac{I_{2}}{I_{1}}\right),
\end{split}
\end{equation}
where 
\begin{equation}
    I_{1}=\int_{\G}\frac{|u|^{p}}{|x|_{CC}^{ap-\beta}}d\lambda(x),
\end{equation}
\begin{equation}
    I_{2}=\int_{\G}\frac{|u|^{\p}}{|x|_{CC}^{\beta}}d\lambda(x),
\end{equation}
and 
\begin{equation}
    I_{3}=\int_{\G}\frac{|u|^{p}}{|x|_{CC}^{ap-\beta}}\log(|x|_{CC})d\lambda(x).
\end{equation}
Then we have
\begin{equation}\label{otkrpravchnon}
    \begin{split}
        &\lim_{r\rightarrow 0}\frac{1}{r}\Biggl{[}\left(\int_{\G}\frac{|u|^{p}}{|x|_{CC}^{ap-\beta\frac{\p-p}{\p-(r+p)}}}d\lambda(x)\right)^{\frac{\p-(r+p)}{\p-p}}\left(\int_{\G}\frac{|u|^{\p}}{|x|_{CC}^{\beta}}d\lambda(x)\right)^{\frac{(r+p)-p}{\p-p}}\\&
        -\left(\int_{\G}\frac{|u|^{p}}{|x|_{CC}^{ap-\beta\frac{\p-p}{\p-p}}}d\lambda(x)\right)^{\frac{\p-p}{\p-p}}\left(\int_{\G}\frac{|u|^{\p}}{|x|_{CC}^{\beta}}d\lambda(x)\right)^{\frac{p-p}{\p-p}}\Biggl{]}\\&
        =\frac{I_{1}}{\p-p}\left(\beta\frac{I_{3}}{I_{1}}+\log\frac{I_{2}}{I_{1}}\right)\\&
        =\frac{\beta}{\p-p}I_{3}+\frac{I_{1}}{\p-p}\frac{\p p}{\p p}\log\frac{I_{2}}{I_{1}}\\&
        =\frac{\beta}{\p-p}I_{3}+\frac{I_{1}\p}{(\p-p)p}\log\frac{I^{\frac{p}{\p }}_{2}}{I_{1}^{1-1+\frac{p}{\p}}}\\&
        =\frac{\beta}{\p-p}I_{3}+\frac{I_{1}\p}{(\p-p)p}\log\frac{I^{\frac{p}{\p }}_{2}}{I_{1}}-\frac{I_{1}\p }{(\p-p)p}\log I_{1}^{-1+\frac{p}{\p}}\\&
        =\frac{\beta}{\p-p}I_{3}+\frac{I_{1}\p}{(\p-p)p}\log\frac{I^{\frac{p}{\p }}_{2}}{I_{1}}+\frac{I_{1}}{p}\log I_{1}.
    \end{split}
\end{equation}

Putting \eqref{otkrpravchnon} and \eqref{otkrlevchnon} in \eqref{welogosn2non}, we have 
\begin{equation}\label{predposnon}
\begin{split}
     \int_{\G}\frac{|u|^{p}\log(|u|^{p}|x|_{CC}^{d-ap})}{|x|_{CC}^{ap-\beta}}&d\lambda(x)\leq \frac{\beta p}{\p-p}I_{3}+\frac{I_{1}\p}{\p-p}\log\frac{I^{\frac{p}{\p}}_{2}}{I_{1}}+I_{1}\log I_{1}\\&
     =\int_{\G}\frac{|u|^{p}}{|x|_{CC}^{ap-\beta}}\log\left(|x|_{CC}^{\frac{\beta(d-ap)}{ap-\beta}}\int_{\G}\frac{|u|^{p}}{|x|_{CC}^{ap-\beta}}d\lambda(x)\right)d\lambda(x)\\&
     +\frac{(d-\beta)\int_{\G}\frac{|u|^{p}}{|x|_{CC}^{ap-\beta}}d\lambda(x)}{ap-\beta}\log\frac{\left(\int_{\G}\frac{|u|^{\p}}{|x|_{CC}^{\beta}}d\lambda(x)\right)^{\frac{p}{\p}}}{\int_{\G}\frac{|u|^{p}}{|x|_{CC}^{ap-\beta}}d\lambda(x)}.
\end{split}
\end{equation}
From Theorem \ref{Harsobinnthm}, if $0\leq \beta<ap<d$, we have the Hardy-Sobolev inequality with $\p=\frac{(d-\beta)p}{d-ap}$ in the following form:
\begin{equation}
    \left\|\frac{u}{|x|_{CC}^{\frac{\beta}{\p}}}\right\|_{L^{\p}(\lambda)}\leq C\|u\|_{L_{a}^{p}(\lambda)}.
\end{equation}
Finally, by using this in \eqref{predposnon}, we get
\begin{equation*}
\begin{split}
    &\int_{\G}\frac{\frac{|u|^{p}}{|x|_{CC}^{ap-\beta}}}{\int_{\G}\frac{|u|^{p}}{|x|_{CC}^{ap-\beta}}d\lambda(x)}\log\left(\frac{|u|^{p}|x|_{CC}^{(d-ap)(1-\frac{\beta}{ap-\beta})}}{\int_{\G}\frac{|u|^{p}}{|x|_{CC}^{ap-\beta}}d\lambda(x)}\right)d\lambda(x) \\ &\leq\frac{d-\beta}{ap-\beta}\log\frac{\left(\int_{\G}\frac{|u|^{\p}}{|x|_{CC}^{\beta}}d\lambda(x)\right)^{\frac{p}{\p}}}{\int_{\G}\frac{|u|^{p}}{|x|_{CC}^{ap-\beta}}d\lambda(x)}\\&
    \leq\frac{d-\beta}{ap-\beta}\log\left(C\frac{\|u\|^{p}_{L^{p}_{a}(\lambda)}}{\int_{\G}\frac{|u|^{p}}{|x|_{CC}^{ap-\beta}}dx}\right),
\end{split}
\end{equation*}
completing the proof.
\end{proof}

\begin{rem}\label{REM:CCR}
In view of \cite[Remark 1.2]{RY18a}, the statement of Theorem \ref{Harsobinnthm} also holds with the Riemannian distance instead of the Carnot-Carath\'eodory distance, so that 
Theorem \ref{THM:weHarnon} also holds with the Riemannian distance instead of the Carnot-Cara\-th\'eo\-dory distance.
\end{rem}

\begin{rem}\label{REM:CCR-nw}
By taking $\beta=\frac{ap}{2}$, as a consequence of \eqref{weloghar1non}, for $0<ap<d$, we obtain an interesting inequality
   \begin{equation}\label{weloghar1non-nw}
         \int_{\mathbb{G}}\frac{|u(x)|^{p}}{|x|_{CC}^{\frac{ap}{2}}}\log\left(|u|\right) d\lambda(x) \leq \frac{2d-ap}{ap^2}\log\left(C\|u\|^{p}_{L_{a}^{p}(\lambda)}\right),
     \end{equation}
    for all $u$ such that $\int_{\G}\frac{|u|^{p}}{|x|_{CC}^{\frac{ap}{2}}}d\lambda(x)=1$.
\end{rem}

\section{Fractional logarithmic Hardy inequality on homogeneous groups}

First, we briefly recall that a Lie group (on $\mathbb{R}^{N}$) $\mathbb{G}$ is called a {\em homogeneous (Lie) group} if it is equipped with the dilations
$$\lambda x:=D_{\lambda}(x):=(\lambda^{\nu_{1}}x_{1},\ldots,\lambda^{\nu_{N}}x_{N}),\; \nu_{1},\ldots, \nu_{n}>0,\; D_{\lambda}:\mathbb{R}^{N}\rightarrow\mathbb{R}^{N},$$
being an automorphism of the group $\mathbb{G}$ for each $\lambda>0.$
We refer to \cite{FS, FR16, RS19} for the extensive description of such groups.
The number 
\begin{equation}
Q:=\nu_{1}+\ldots+\nu_{N},
\end{equation}
is called the homogeneous dimension of $\mathbb{G}$. 
If $|S|$ denotes the volume of a measurable set $S\subset \mathbb{G}$ with respect to the Haar measure $dx$ on $\mathbb{G}$, then 
\begin{equation}\label{scal}
|D_{\lambda}(S)|=\lambda^{Q}|S| \quad {\rm and}\quad \int_{\mathbb{G}}f(\lambda x)
dx=\lambda^{-Q}\int_{\mathbb{G}}f(x)dx.
\end{equation}
We note that homogeneous groups have to be nilpotent, so that they are unimodular, and the standard Lebesgue measure $dx$ on $\mathbb{R}^{N}$ is the Haar measure on $\G$, see e.g. \cite[Proposition 1.6.6]{FR16}.

Homogeneous quasi-norms on $\G$ are continuous non-negative functions
\begin{equation}
\mathbb{G}\ni x\mapsto |x|\in[0,\infty),
\end{equation}
such that
\begin{itemize}
\item[a)] $|x|=|x^{-1}|$ for all $x\in\mathbb{G}$,
\item[b)] $|\lambda x|=\lambda|x|$ for all $x\in \mathbb{G}$ and $\lambda>0$,
\item[c)] $|x|=0$ if and only if $x=0$.
\end{itemize}
We refer to 
\cite[Definition 3.1.33]{FR16} or \cite[Definition 1.2.1]{RS19} and subsequent discussions for their main properties. 

Let us now give the definition of the fractional Sobolev space on homogeneous Lie groups. 
Let $p\geq1$. Then for any measurable function $u:\mathbb{G}\rightarrow \mathbb{R}$, we define its Gagliardo quasi-seminorm in the following form:
\begin{equation}\label{gsmnm}
[u]_{s,p}:=\left( \int_{\mathbb{G}} \int_{\mathbb{G}}\frac{|u(x)-u(y)|^{p}}{|y^{-1} x|^{Q+sp}}dxdy\right)^{\frac{1}{p}},\,\,\,\,s\in(0,1),\,\,Q>1,
\end{equation}
where $|\cdot|$ is some (fixed) quasi-norm on $\G$.
For $p\geq1$ and $s\in(0,1)$, we denote by $W^{s,p}(\mathbb{G})$ the corresponding space, called the fractional Sobolev space on the homogeneous group $\G$, defined by
\begin{equation}
W^{s,p}(\mathbb{G})=\{u: u\in L^{p}(\mathbb{G}), [u]_{s,p}<+\infty \}.
\end{equation}
Then, we define the fractional $p$-sub-Laplacian on $\G$ as follows: for a (Haar) measurable and compactly supported function $u$, the fractional $p$-sub-Laplacian $(-\Delta_{p})^{s}$ on $\mathbb{G}$ is given by
\begin{equation}
(-\Delta_{p})^{s}u(x)=2\lim_{\delta\searrow 0}\int_{\mathbb{G} \setminus B(x,\delta)}\frac{|u(x)-u(y)|^{p-2}(u(x)-u(y))}{|y^{-1} x|^{Q+sp}}dy, \,\,\,x\in \mathbb{G},
\end{equation}
where $|\cdot|$ is a quasi-norm on $\mathbb{G}$, and $B(x,\delta)$ is a quasi-ball with respect to $|\cdot|$, with radius $\delta$ centred at $x\in\mathbb{G}$. 

We recall that the logarithmic Sobolev inequalities for the fractional $p$-sub-Laplacian on homogeneous groups were established in \cite{KRS20}.
We now record the version of the fractional logarithmic Hardy inequality on homogeneous groups:

\begin{thm}\label{REM:fractional}
Let $\G$ be a homogeneous Lie group with homogeneous dimension $Q\geq 3$. Let $p>1$ and $s\in (0,1)$ be such that $0\leq \beta<sp<Q$. Then we have the weighted logarithmic Hardy inequality 
 \begin{equation*}\label{welogharhom}
         \int_{\mathbb{G}}\frac{\frac{|u(x)|^{p}}{|x|^{sp-\beta}}}{\left\|\frac{u}{|x|^{\frac{sp-\beta}{p}}}\right\|^{p}_{L^{p}(\G)}}\log\left(\frac{|x|^{(Q-sp)(1-\frac{\beta}{sp-\beta})}|u|^{p}}{\left\|\frac{u}{|x|^{\frac{sp-\beta}{p}}}\right\|^{p}_{L^{p}(\G)}}\right) dx \leq \frac{Q-\beta}{sp-\beta}\log\left(C\frac{[u]^{p}_{s,p}}{\left\|\frac{u}{|x|^{\frac{sp-\beta}{p}}}\right\|^{p}_{L^{p}(\G)}}\right),
     \end{equation*}
     for all $u\in W^{s,p}(\G)\backslash\{0\}$ with $|x|^{-\frac{sp-\beta}{p}}u\in L^{p}(\G)$.
\end{thm}
 The proof of this theorem is similar to the proof of Theorem \ref{THM:weHarnon}, but here instead of Theorem \ref{THM:HSgraded} we use the fractional Hardy-Sobolev inequality in \cite[Theorem 3.3]{KS20}, so we omit the repetition of the details. 

\section{Logarithmic Hardy-Rellich inequalities on graded groups}
\label{SEC:3}

In this section we give a refinement of the logarithmic Hardy-Rellich inequality from Theorem \ref{THM:weHarnon} in the setting of graded groups. 

We recall from e.g. \cite[Definition 3.1.1]{FR16}, that a Lie algebra $\mathfrak{g}$ is
called {\em graded} if it is endowed with a vector space decomposition 
\begin{equation}
\mathfrak{g}=\oplus_{j=1}^{\infty}V_{j},\,\,\,\,\text{such that}\,\,\,[V_{i},V_{j}]\subset V_{i+j},
\end{equation}
where all but
finitely many of the $V_{j}$'s are $\{0\}$.
Consequently, a connected and simply-connected
Lie group is called {\em graded} if its Lie algebra is graded.

We call left-invariant homogeneous hypoelliptic differential operators on $\G$ by {\em Rockland operators}. An alternative characterisation (due to Helffer and Nourigat \cite{HN79}) of Rockland operators can be given as follows.
Let $\widehat{\G}$ be the unitary dual of $\G$, and for $\pi\in \widehat{\G}$, let $\mathcal{H}_{\pi}^{\infty}$ be the space of the smooth vectors of $\pi$. 
Let $A$ be a left-invariant differential operator on $\G$.
We say that $A$ satisfies the {\em Rockland condition} when
for each nontrivial representation $\pi\in\widehat{\G}$, 
the operator $\pi(A)$ is injective on $\mathcal{H}_{\pi}^{\infty}$, that is,
\begin{equation}
\forall v \in\mathcal{H}_{\pi}^{\infty},\,\,\, \pi(A)v =0 \Rightarrow v = 0.
\end{equation}
Then Rockland operators can be characterised by being left-invariant homogeneous differential operators which satisfy the Rockland condition.
We refer to 
\cite[Definitions 4.1.1 and 4.1.2]{FR16} for more discussion of such operators but note that if a homogeneous Lie group admits a Rockland operator, it must be graded.

Let $\mathcal{R}$ be a Rockland operator of homogeneous degree $\nu$. We can define the homogeneous Sobolev space by the norm
\begin{equation}
    \|f\|_{\dot{L}^p_{a}(\G)}:=\|\R^{\frac{a}{\nu}}f\|_{L^{p}(\G)},
\end{equation}
and its inhomogeneous version by 
\begin{equation}
    \|f\|_{L^p_{a}(\G)}:=(\|f\|^{p}_{L^{p}(\G)}+\|\R^{\frac{a}{\nu}}f\|^{p}_{L^{p}(\G)})^{\frac{1}{p}}.
\end{equation}
While these norms clearly depend on the choice of $\mathcal{R}$, it was 
shown in \cite{FR16, FR17} that the Sobolev spaces defined by these norms do not depend on a particular choice of the Rockland operator $\R$. 

Our next aim is to give a weighted version of the log-Sobolev inequality. For this let us recall the so-called Hardy-Sobolev family of inequalities from 
 \cite[Theorem 5.1]{RY18b}. \begin{thm}[Hardy-Sobolev inequality]\label{THM:HSgraded}
Let $\G$ be a graded Lie group of homogeneous dimension $Q$ and let
$\R$ be a positive Rockland operator of homogeneous degree $\nu$. Let $|\cdot|$ be an arbitrary homogeneous quasi-norm. Let $1<p\leq q<\infty$. Let $0<a p<Q$ and $0\leq \beta< Q$. Assume that $\frac{1}{p}-\frac{1}{q}=\frac{a}{Q}-\frac{\beta}{qQ}$. Then there exists a positive constant $C$ such that
\begin{equation}\label{EQ:HSgrad}
    \left\|\frac{u}{|x|^{\frac{\beta}{q}}}\right\|_{L^{q}(\G)}\leq C\|\R^{\frac{a}{\nu}}u\|_{L^{p}(\G)},
\end{equation}
holds for all $u\in \dot{L}^{p}_{a}(\G)$.
\end{thm}

Using Theorem \ref{THM:HSgraded} instead of Theorem \ref{Harsobinnthm}, we can repeat the proof of Theorem \ref{THM:weHarnon}. This will give the following weighted Hardy inequality on graded groups:
\begin{thm}\label{THM:weHargr}
Let $\G$ be a graded Lie group of homogeneous dimension $Q$ and let
$\R$ be a positive Rockland operator of homogeneous degree $\nu$. Let $|\cdot|$ be an arbitrary homogeneous quasi-norm. Let $1<p<\infty$ and $0\leq \beta<ap<Q.$ Then we have 
    \begin{equation*}\label{weloghar}
         \int_{\mathbb{G}}\frac{\frac{|u(x)|^{p}}{|x|^{ap-\beta}}}{\left\|\frac{u}{|x|^{\frac{ap-\beta}{p}}}\right\|^{p}_{L^{p}(\G)}}\log\left(\frac{|x|^{(Q-ap)(1-\frac{\beta}{ap-\beta})}|u|^{p}}{\left\|\frac{u}{|x|^{\frac{ap-\beta}{p}}}\right\|^{p}_{L^{p}(\G)}}\right) dx \leq \frac{Q-\beta}{ap-\beta}\log\left(C\frac{\|\R^{\frac{a}{\nu}}u\|^{p}_{L^{p}(\G)}}{\left\|\frac{u}{|x|^{\frac{ap-\beta}{p}}}\right\|^{p}_{L^{p}(\G)}}\right),
     \end{equation*}
     for all nontrivial $u$ for which the right hand side makes sense. 
   In particular, for all $u$ such that $\int_{\G}\frac{|u|^{p}}{|x|^{ap-\beta}}dx=1$, we have 
   \begin{equation}\label{weloghar1}
         \int_{\mathbb{G}}\frac{|u(x)|^{p}}{|x|^{ap-\beta}}\log\left(|x|^{(Q-ap)(1-\frac{\beta}{ap-\beta})}|u|^{p}\right) dx \leq \frac{Q-\beta}{ap-\beta}\log\left(C\|\R^{\frac{a}{\nu}}u\|^{p}_{L^{p}(\G)}\right).
     \end{equation}
\end{thm}
We will omit the proof as it is almost verbatim repetition of the proof of Theorem \ref{THM:weHarnon}, with $d$ replaces by $Q$, and by using 
Theorem \ref{THM:HSgraded} instead of Theorem \ref{Harsobinnthm}.
The refinement with respect to now using the homogeneous Sobolev norm is essential for some applications, for example for recovering and extending the Gross logarithmic inequalities to the setting of stratified groups in the next section.

\section{Gross type logarithmic Hardy inequality on stratified groups}\label{SEC:GrossHardy}

In this section we show that the established logarithmic Hardy inequalities imply the Gross type logarithmic Hardy inequalities, in the setting of stratified groups.

We very briefly recall the basics, referring e.g. to the open access books \cite{FR16} and \cite{RS19} for further details, as well as to the book \cite{BLU07}. 
Thus, we call the Lie group $\mathbb{G}=(\mathbb{R}^{n},\circ)$ a {\em stratified  group} if it satisfies the following properties:

(a) for some integers $n_{1}+...+n_{r}=n$ we have 
the decomposition $\mathbb{R}^{n}=\mathbb{R}^{n_{1}}\times...\times\mathbb{R}^{n_{r}}$ such that
for every $\lambda>0$ the dilation $\delta_{\lambda}: \mathbb{R}^{n}\rightarrow \mathbb{R}^{n}$
given by
$$\delta_{\lambda}(x)\equiv\delta_{\lambda}(x^{(1)},...,x^{(r)}):=(\lambda x^{(1)},...,\lambda^{r}x^{(r)})$$
is an automorphism of the group $\mathbb{G},$ where $x^{(k)}\in \mathbb{R}^{n_{k}}$ for $k=1,...,r.$

(b) if $n_{1}$ is as in (a) and $X_{1},...,X_{n_{1}}$ are the left invariant vector fields on $\mathbb{G}$ such that
$X_{k}(0)=\frac{\partial}{\partial x_{k}}|_{0}$ for $k=1,...,n_{1}$, then
$${\rm rank}({\rm Lie}\{X_{1},...,X_{n_{1}}\})=n,$$
for every $x\in\mathbb{R}^{n}.$ This means that the iterated commutators
of $X_{1},...,X_{n_{1}}$ span the Lie algebra of $\mathbb{G}.$

A Lie algebra $\mathfrak{g}$ is called stratified if it has a decomposition
\begin{equation}
\mathfrak{g}=\oplus_{j=1}^{\infty}V_{j},\,\,\,\,\text{such that}\,\,\,[V_{i},V_{1}]\subset V_{i+1},
\end{equation}
with $V_{1}$ generating $\mathfrak{g}$ as an algebra. In particular, any stratified group is also graded. 
The natural dilations on $\mathfrak{g}$ are given by
\begin{equation}
    D_{r}X_{k}=r^{k}X_{k},\,\,\,\,(X_{k}\in V_{k},\; k=1,\ldots,m),
\end{equation}
where $m$ is the step of $\mathfrak{g}$ (the number of nontrivial iterated commutators). 

We now record the following corollary of Theorem \ref{THM:weHargr} in the case of $\R$ being the positive sub-Laplacian on a stratified group, $p=2$ and $a=1$.
\begin{cor}\label{cor.log.Hardy21}
Let $\G$ be a stratified Lie group of homogeneous dimension $Q$ and let $|\cdot|$ be an arbitrary homogeneous quasi-norm. Then, the following weighted logarithmic Hardy inequality is satisfied
\begin{equation}
    \label{log.Hard21}
    \int_{\G} \frac{|u(x)|^2}{|x|^{2-\beta}} \log \left(|x|^{\frac{(Q-2)(1-\beta)}{2-\beta}}|u(x)| \right)\,dx \leq \frac{Q-\beta}{2(2-\beta)} \log \left(C \|\nabla_H u\|^{2}_{L^2(\G)} \right)\,,
\end{equation}
for every $0\leq \beta <2<Q$, and for every $u$ such that $\left\|\frac{u}{|x|^{\frac{2-\beta}{2}}} \right\|_{L^2(\G)}=1$.
\end{cor}
In the next theorem we show that the weighted logarithmic Hardy inequality as in \eqref{log.Hard21} implies its Gross counterpart.

\begin{thm}\label{thm.gaus.Hardy}
Let $\G$ be a stratified group with homogeneous dimension $Q$, topological dimension $n$, and let $n_1$ be the dimension of the first stratum of its Lie algebra, i.e., for $x \in \G$ we can write $x=(x',x'') \in \mathbb{R}^{n_1} \times \mathbb{R}^{n-n_1}$. 
Let $|\cdot|$ be an arbitrary homogeneous quasi-norm on $\G$. 
Let $|x'|$ denote the Euclidean norm of $x'$, and let $M>0$ be a constant such that we have 
\begin{equation}\label{EQ:norms}
|x'|\leq M|x|,
\end{equation}
for the quasi-norm $|x|$, for all $x\in\G.$
Then the following weighted ``semi-Gaussian'' logarithmic Hardy inequality is satisfied
\begin{equation}
    \label{gaus.log.Hard}
    \int_{\G}\frac{|g(x)|^2}{|x|^{2-\beta}} \log \left(|x|^{\frac{(Q-2)(1-\beta)}{2-\beta}}|g(x)| \right)\,d\mu \leq \int_{\G} |\nabla_{H}g|^2\,d\mu\,,
\end{equation}
for all $0\leq \beta <2<Q$, and for all $g$ such that $\left\|\frac{g}{|x|^{\frac{2-\beta}{2}}} \right\|_{L^2(\mu)}=1,$
 where $\mu=\mu_1 \otimes \mu_2$, and $\mu_1$ is the Gaussian measure on $\mathbb{R}^{n_1}$ given by $d\mu_1=\gamma e^{-\frac{|x'|^2}{2}}dx'$, for $x'\in \mathbb{R}^{n_1}$, where the normalisation constant $\gamma$ is given by 
 \begin{equation}
     \label{k-w}
     \gamma:= 
     \left(\frac{Q-\beta}{2(2-\beta)}Ce^{\left(n_1+\frac{M^2}{2} \right)\left(\frac{2-\beta}{Q-\beta} \right)-1} \right)^{\frac{Q-\beta}{2-\beta}}
 \end{equation} 
 and $\mu_2$ is the Lebesgue measure $dx''$ on $\mathbb{R}^{n-n_1}$.
\end{thm}
Let us note that the constant $M=\max\limits_{|x|=1} |x'|$ in \eqref{EQ:norms} exists for any homogeneous quasi-norm due to continuity and homogeneity. Moreover, the appearance of $M$ at the normalisation constant $\gamma$ as in \eqref{k-w} indicates the dependence of the right-hand side of \eqref{gaus.log.Hard} on the quasi-norm $|\cdot|$ as one expects.
\begin{proof}
Assume that $g$ is such that $\left\| \frac{g}{|x|^{\frac{2-\beta}{2}}}\right\|_{L^2(\mu)}=1$ , where $\mu$ is as in the hypothesis. Defining $u$ by 
\[
u(x)=\gamma^{1/2}e^{-\frac{|x'|^2}{4}}g(x)\,,
\]
for $\gamma$ as in \eqref{k-w}, we can check that $u \in L^2(\G)$, and in particular we have $\left\|\frac{u}{|x|^{\frac{2-\beta}{2}}} \right\|_{L^2(\G)}=1$. Indeed, rigorous computations show that \[
1=\int_{\G}\frac{|g(x)|^2}{|x|^{2-\beta}}\,d\mu=\int_{\G}\gamma^{-1}e^{\frac{|x'|^2}{2}}\frac{|u(x)|^2}{|x|^{2-\beta}}\,d\mu=\int_{\G}\frac{|u(x)|^2}{|x|^{2-\beta}}\,dx\,.
\]
Applying the logarithmic Holder inequality \eqref{log.Hard21} to $u$, we arrive at 
\begin{eqnarray}
\label{log.Hold,u}
\int_{\G}\frac{|g(x)|^2}{|x|^{2-\beta}} \log \left(|x|^{\frac{(Q-2)(1-\beta)}{2-\beta}}|g(x)| \right)\,d\mu & \leq & \int_{\G}\frac{|u(x)|^2}{|x|^{2-\beta}} \log\left(\gamma^{-1/2}e^{\frac{|x'|^2}{4}}|x|^{\frac{(Q-2)(1-\beta)}{2-\beta}}u(x) \right)\,dx\nonumber\\
& \leq & \frac{Q-\beta}{2(2-\beta)} \log\left(C \int_{\G}|\nabla_H u(x)|^2\,dx \right)\nonumber\\
&+&\log(\gamma^{-1/2})+\int_{\G}\frac{|x'|^2}{4}\frac{|u(x)|^2}{|x|^{2-\beta}}\,dx\,.
\end{eqnarray}
Using \eqref{EQ:norms} we additionally estimate
\begin{multline}\label{EQ:aterm}
 \int_{\G}\frac{|x'|^2}{4}\frac{|u(x)|^2}{|x|^{2-\beta}}\,dx=\int_{|x|\leq 1} +\int_{|x|\geq 1}  \\ \leq
 \frac{M^2}{4}\int_{|x|\leq 1} \frac{|u(x)|^2}{|x|^{2-\beta}}\,dx+
 \int_{|x|\geq 1}
 \frac{|x'|^2}{4} |u(x)|^2\,dx
 \\ \leq \frac{M^2}{4}+\int_\G  \frac{|x'|^2}{4} |u(x)|^2\,dx\,.
\end{multline}
To estimate the term 
\[
|\nabla_H g(x)|=\sqrt{\sum_{i=1}^{n_1}|X_ig(x)|^2}\,,
\]
recall (see e.g. \cite{FR16}) that the vector fields $X_i$ for $i=1,\ldots,n_1$, are given by 
\[
X_i=\partial_{x_{i}^{'}}+\sum_{j=1}^{n-n_1}p_{j}^{i}(x')\partial_{x''_{j}}\,.
\]
For $i=1,\ldots,n_1$ and for $g$ as above, we compute
\begin{eqnarray}\label{Xg}
|X_ig(x)|^2 &=& \gamma^{-1} e^{\frac{|x'|^2}{2}} \left|X_iu(x)+\frac{x'_{i}}{2}u(x) \right|^{2}\nonumber\\
&=& \gamma^{-1} e^{\frac{|x'|^2}{2}} \left(|X_iu(x)|^2+\frac{(x'_{i})^2}{4}|u(x)|^2+{\rm Re} \overline{(X_iu(x))}x'_{i}u(x) \right)\,.
\end{eqnarray}
Moreover, for $x'_{i}$, $i=1,\ldots,n_1$, we have
\begin{eqnarray*}
{\rm Re}\int_{\G}\overline{(\partial_{x'_{i}}u(x))}x'_{i}u(x)\,dx & =& -{\rm Re}\int_{\G}(\partial_{x'_{i}}u(x))x'_{i}\overline{u(x)}\,dx-\int_{\G}|u(x)|^2\,dx\\
&=& -{\rm Re}\int_{\G}\overline{(\partial_{x'_{i}}u(x))}x'_{i}u(x)\,dx-1\,,
\end{eqnarray*}
where we used the integration by parts. This gives
\begin{equation}
    \label{int.parts1}
    {\rm Re}\int_{\G}\overline{(\partial_{x'_{i}}u(x))}x'_{i}u(x)\,dx=-\frac{1}{2}\,.
\end{equation}
Similar arguments, for $j=1,\ldots,n-n_1$ and for $i=1,\ldots,n_1$, give
\begin{eqnarray*}
   {\rm Re} \int_{\G}p_{j}(x')\overline{(\partial_{x_{j}''}u(x))}x'_{i}u(x)& = & - {\rm Re}\int_{\G}\partial_{x_{j}''}((p_{j}(x')u(x)x'_{i})\overline{u(x)}\,dx\\
    &=&- {\rm Re}\int_{\G}p_{j}(x')\overline{(\partial_{x_{j}''}u(x))}x'_{i}u(x)\,dx\,,
\end{eqnarray*}
or equivalently
\begin{equation}  \label{int.parts2}
   {\rm Re} \int_{\G}p_{j}(x')\overline{(\partial_{x_{j}''}u(x))}x'_{i}u(x)\,dx=0\,.
\end{equation}
Combining \eqref{int.parts1} and \eqref{int.parts2} we arrive at 
\[
\int_{\G}{\rm Re}\overline{(X_iu(x))}x'_{i}u(x)\,dx=-\frac{1}{2}\,, \quad \forall i=1,\ldots,n_1\,.
\]
Plugging \eqref{int.parts1} and \eqref{int.parts2} into \eqref{Xg} we get the equality
\begin{equation}\label{thm.eq.nablag}
\int_{\G}|\nabla_{H}g(x)|^2\,d\mu=\int_{\G}|\nabla_{H}u(x)|^2\,dx+\int_{\G}\frac{|x'|^2}{4}|u(x)|^2\,dx-\frac{n_1}{2}\,.
\end{equation}
Combining inequalities \eqref{log.Hold,u}, \eqref{EQ:aterm} and \eqref{thm.eq.nablag}, we see that to prove the desired inequality \eqref{gaus.log.Hard} it is enough to show that 
\[
\frac{Q-\beta}{2(2-\beta)} \log\left(C \int_{\G} |\nabla_H u(x)|^2\,dx \right)+\log (\gamma^{-1/2})+\frac{M^2}{4} \leq \int_{\G} |\nabla_H u(x)|^2\,dx-\frac{n_1}{2}\,,
\]
which can be equivalently written as 
\[
\log\left(C\gamma^{-\frac{2-\beta}{Q-\beta}}e^{\left(n_1+\frac{M^2}{2}\right)\left(\frac{2-\beta}{Q-\beta} \right)}\int_{\G}|\nabla_H u(x)|^2\,dx \right) \leq \frac{2(2-\beta)}{Q-\beta}\int_{\G}|\nabla_H u(x)|^2\,dx\,.
\] 
Now, since $\log r \leq r-1$, for all $r>0$, it suffices to show that 
\[
e^{-1}C\gamma^{-\frac{2-\beta}{Q-\beta}}e^{\left(n_1+\frac{M^2}{2}\right)\left(\frac{2-\beta}{Q-\beta} \right)}\int_{\G}|\nabla_H u(x)|^2\,dx \leq \frac{2(2-\beta)}{Q-\beta} \int_{\G} |\nabla_H u(x)|^2\,dx\,,
\]
where the last is satisfied as an equality for $\gamma$ as in \eqref{k-w}, and we have
\begin{eqnarray*}
&&\log\left( C\gamma^{-\frac{2-\beta}{Q-\beta}}e^{\left(n_1+\frac{M^2}{2}\right)\left(\frac{2-\beta}{Q-\beta} \right)}\int_{\G}|\nabla_H u(x)|^2\,dx\right)\\
&=&\log\left(e^{-1}C\gamma^{-\frac{2-\beta}{Q-\beta}}e^{\left(n_1+\frac{M^2}{2}\right)\left(\frac{2-\beta}{Q-\beta} \right)}\int_{\G}|\nabla_H u(x)|^2\,dx \right)+1\\
&\leq & \frac{2(2-\beta)}{Q-\beta}\int_{\G}|\nabla_H u(x)|^2\,dx\,,
\end{eqnarray*}
completing the proof.
\end{proof}
As a particular case of Theorem \ref{thm.gaus.Hardy}, we can record its unweighted version: 
\begin{cor}
Let $\G$ be a stratified group with homogeneous dimension $Q$.
The following ``semi-Gaussian'' Gross type logarithmic Hardy inequality is satisfied
\begin{equation}  \label{eq:spb}
\int_{\G}\frac{|g(x)|^2}{|x|^2} \log \left(|x|^{\frac{Q-2}{2}}|g(x)| \right)\,d\mu \leq \int_{\G}|\nabla_{H}g|^2\,d\mu\,,
\end{equation}
for all $g$ such that $\left\|\frac{g}{|x|} \right\|_{L^2(\mu)}=1$, where $\mu$ is as in Theorem \ref{thm.gaus.Hardy} and the normalisation constant $\gamma$ is now given by 
\[
\gamma:= \left(4^{-1}Qe^{\frac{2n_1+M^2}{Q}-1}C \right)^{\frac{Q}{2}}.
\]
By taking $\beta=1$, we get the following interesting inequality:
\begin{equation}  \label{eq:spb1}
\int_{\G}\frac{|g(x)|^2}{|x|} \log \left(|g(x)| \right)\,d\mu \leq \int_{\G}|\nabla_{H}g|^2\,d\mu\,,
\end{equation}
for all $g$ such that $\left\|\frac{g}{|x|^{1/2}} \right\|_{L^2(\mu)}=1$.
\end{cor}

\section{Logarithmic Poincar\'{e} inequality}

In this section, we briefly record a logarithmic Poincar\'{e} inequality on stratified groups. The proof is simple if we recall the logarithmic H\"{o}lder inequality on general measure spaces from \cite{CKR21b}:
\begin{lem}[Logarithmic H\"older inequality]\label{holder}
Let $\X$ be a  measure space. Let $u\in L^{p}(\mathbb{X})\cap L^{q}(\mathbb{X})\setminus\{0\}$ with some $1<p<q< \infty.$ 
Then we have
\begin{equation}\label{holdernn}
\int_{\mathbb{X}}\frac{|u|^{p}}{\|u\|^{p}_{L^{p}(\mathbb{X})}}\log\left(\frac{|u|^{p}}{\|u\|^{p}_{L^{p}(\mathbb{X})}}\right)dx\leq \frac{q}{q-p}\log\left(\frac{\|u\|^{p}_{L^{q}(\mathbb{X})}}{\|u\|^{p}_{L^{p}(\mathbb{X})}}\right).
\end{equation}
\end{lem}
Let us denote average value of the function $u$ by
\begin{equation}
    (u)_{R}:=\frac{1}{|B_{R}|}\int_{B_{R}}u(x)dx,
\end{equation}
where $B_{R}=B(e,R)$ is the CC-ball centred at the unit element $e$  with radius $R$. By $(u)_{\infty}$ we denote the limit of the
average value $(u)_{R}$  of $u$ on the ball $B_R$  as $R$ goes to
infinity.
Let us also recall the Poincar\'{e} inequality on stratified groups from \cite{CLW07}.
\begin{thm}[{\cite[Theorem 2.1]{CLW07}}]\label{THM:Lu}
Let $\G$ be a stratified Lie group with homogeneous dimension $Q$. Assume that $1\leq p< Q$, $f\in S^{1,p}_{loc}(\G)$ is in the local Sobolev space with horizontal derivatives in $L^p$, and $|\nabla_{H}u|\in L^{p}(\G)$. As $R$ approaches to infinity, $(u)_{R}$ converges to a finite limit $(u)_{\infty}$. Moreover,
\begin{equation}\label{poincareineq}
    \|u-(u)_{\infty}\|_{L^{p^{*}}(\G)}\leq C\|\nabla_{H}u\|_{L^{p}(\G)},\,\,\,\,\,p^{*}=\frac{Qp}{Q-p},
\end{equation}
where  $C(p,Q)$ is a positive constant depending on $p,Q$ only.
\end{thm}
Let us show now that a combination of Lemma \ref{holder} and Theorem \ref{THM:Lu} immediately imply the
logarithmic Poincar\'{e} inequality on stratified groups:
\begin{thm}
Let $\mathbb{G}$ be a stratified group with homogeneous dimension $Q$.
Assume that $1<p<Q$. Then we have 
\begin{equation*}\label{LogPoincare}
\int_{\mathbb{G}}\frac{|u-(u)_{\infty}|^{p}}{\|u-(u)_{\infty}\|^{p}_{L^{p}(\mathbb{G})}}\log\left(\frac{|u-(u)_{\infty}|^{p}}{\|u-(u)_{\infty}\|^{p}_{L^{p}(\mathbb{G})}}\right) dx \leq \frac{Q}{p}\log\left(C\frac{\|\nabla_{H}u\|^{p}_{L^{p}(\G)}}{\|u-(u)_{\infty}\|^{p}_{L^{p}(\mathbb{G})}}\right)
\end{equation*}
with the constant $C=C(p,Q)>0$, for any non-trivial $u\in L^p_1(\G)$. 
\end{thm}
\begin{proof}
By combining the logarithmic H\"older inequality \eqref{holdernn} with $1< p<q=p^{*}=\frac{Qp}{Q-p}$ and the Poincar\'{e} inequality \eqref{poincareineq}, we have
\begin{equation*}\label{logpoin}
\begin{split}
    \int_{\mathbb{G}}\frac{|u-(u)_{\infty}|^{p}}{\|u-(u)_{\infty}\|^{p}_{L^{p}(\mathbb{G})}}\log\left(\frac{|u-(u)_{\infty}|^{p}}{\|u-(u)_{\infty}\|^{p}_{L^{p}(\mathbb{G})}}\right) dx  &\leq  \frac{p^{*}}{p^{*}-p}\log\left(\frac{\|u-(u)_{\infty}\|^{p}_{L^{q}(\mathbb{G})}}{\|u-(u)_{\infty}\|^{p}_{L^{p}(\mathbb{G})}}\right)\\&
\stackrel{\eqref{poincareineq}}\leq  \frac{p^{*}}{p^{*}-p} \log\left(C\frac{\|\nabla_{H}u\|^{p}_{L^{p}(\G)}}{\|u-(u)_{\infty}\|^{p}_{L^{p}(\mathbb{G})}}\right)\\&
=\frac{Q}{p}\log\left(C\frac{\|\nabla_{H}u\|^{p}_{L^{p}(\G)}}{\|u-(u)_{\infty}\|^{p}_{L^{p}(\mathbb{G})}}\right),
\end{split}
\end{equation*}
completing the proof.
\end{proof}


\begin{thebibliography}{BBGM12}
\bibitem[Ada79]{Ad79}
R. A. Adams.
\newblock General Logarthmic Sobolev Inequalities and Orlicz Imbeddings.
\newblock {\em J. Funct. Anal.}, 34:292--303, 1979.

\bibitem[AC79]{AC79}
R. A. Adams and F. H. Clarke.
\newblock Gross's Logarithmic Sobolev Inequality: A Simple Proof.
\newblock {\em Amer. J. Math.}, 101(6):1265--1269, 1979.


\bibitem[AR20]{AR20}
R. Akylzhanov and M. Ruzhansky.
\newblock $L^p$-$L^q$ multipliers on locally compact groups.
\newblock {\em J. Funct. Anal.}, 278(108324):1--49, 2020.

\bibitem[AB00]{AB00}
C. An\'e, S. Blach\'ere, D. Chafa\"i, P. Foug\'eres, I. Gentil, F. Malrieu, C. Roberto and G. Scheffer.
\newblock {\em Sur les in\'egalites de Sobolev logarithmiques.} With a preface by D. Bakry and M. Ledoux.
\newblock  Panoramas et Synth\'eses, 10, {\em Soc. Math. Fr, Paris}, 2000. xvi+217 pp.













\bibitem[BLU07]{BLU07}
A. Bonfiglioli, E. Lanconelli and F. Uguzzoni. {\em Stratified Lie Groups and Potential Theory for their Sub-Laplacians.} Springer, Berlin, 2007.


\bibitem[BZ21]{BZ21a}
E. Bou Dagher and B. Zegarlinski.
\newblock Coercive Inequalities on Carnot Groups: Taming Singularities.
\newblock {\em arXiv:2105.03922}, 2021.







\bibitem[BPTV19]{BPTV19}
 T. Bruno, M. Peloso, A. Tabacco and M. Vallarino.  Sobolev spaces on Lie groups: embedding theorems and algebra properties. {\em J. Funct. Anal.}, 276(10):3014--3050, 2019.
 
 \bibitem[BPV21]{BPV21}
 T. Bruno, M. Peloso and M. Vallarino.
 \newblock The Sobolev emedding constant on Lie groups.
 \newblock {\em arXiv:2006.07056v2}, 2021.
 
 


\bibitem[CFZ21]{CFZ21}
M. Chatzakou, S. Federico and B. Zegarlinski.
\newblock $q$-Poncar\'e inequalities on Carnot groups with filiform type Lie algebra.
\newblock {\em arXiv:2007.04689v3}, 2021.

\bibitem[CKR21a]{CKR21a}
M. Chatzakou, A. Kassymov and M. Ruzhansky. Anisotropic Shannon inequality. {\em arXiv:2106.14182}, 2021.

\bibitem[CKR21b]{CKR21b}
M. Chatzakou, A. Kassymov and M. Ruzhansky. Logarithmic Sobolev inequalities on Lie groups. {\em arXiv:2106.15652}, 2021.

\bibitem[CLW07]{CLW07}
W.S. Cohn, G. Lu and P. Wang. Sub-elliptic global high order Poincar\'{e} inequalities in stratified Lie groups and applications. {\em  J.  Funct.  Anal.}, 249(2):393-424, 2007.

\bibitem[Dav89]{Dav89}
E. B. Davies.
\newblock {\em Heat kernels and spectral theory}. Cambridge Tracts in Mathematics, vol. 92, Cambridge University Press, Cambridge, 1989.

\bibitem[Das21]{Das21}
U. Das. On weighted logarithmic-Sobolev \& logarithmic-Hardy inequalities. 
{\em J. Math. Anal. Appl.,} 496(1):124796, 2021.

\bibitem[DD03]{DD03}
M. Del Pino and J. Dolbeault.
\newblock The optimal Euclidean $L^p$-Sobolev logarithmic inequality.
\newblock {\em J. Funct. Anal.}, 197(1):151--161, 2003.

\bibitem[DDFT10]{DDFT10}
M. Del Pino, J. Dolbeault, S. Filippas and A. Tertikas.
\newblock 
A logarithmic Hardy inequality.
\newblock 
{\em J. Funct. Anal.}, 259(8):2045--2072, 2010.


\bibitem[FNQ18]{FNQ18}
T.~Feng, P.~Niu and J.~Qiao.
\newblock Several logarithmic Caffarelli-Kohn-Nirenberg inequalities and applications.
\newblock{\em J. Math. Anal. Appl.}, 457:822--840, 2018.

\bibitem[FR17]{FR17}
V. Fischer and M. Ruzhansky. Sobolev spaces on graded Lie groups. {\em Annales de l'Institut Fourier} 67(4):1671-1723, 2017.

\bibitem[FR16]{FR16}
V. Fischer and M. Ruzhansky. {\em Quantization on nilpotent Lie groups}. Progress in Mathematics, Vol. 314, Birkh\"{a}user, 2016. 

\bibitem[FS82]{FS}
G. Folland and E. M. Stein.
\newblock {\em Hardy spaces on homogeneous groups}.
\newblock Vol. 28 of Math. Notes, Princeton University Press, Princeton, N.J., 1982.



\bibitem[Gro75]{Gro75}
L. Gross.
\newblock Logarithmic Sobolev Inequalities.
\newblock {\em Amer. J. Math.}, 97(4):1061--1083, 1975.

\bibitem[GZ03]{GZ03}
A. Guionnet and B. Zegarlinski.
\newblock Lectures on logarithmic Sobolev inequalities.
\newblock {\em S\'eminaire de Probabilit\'es, XXXVI}, Lect. not. Math., 1801:1--134, 2003.

\bibitem[HMM05]{HMM05}
W. Hebisch, G. Mauceri and S. Meda. Spectral multipliers for Sub-Laplacians with drift on Lie groups. {\em Math. Z.}, 251(4):899--927, 2005.

\bibitem[HZ10]{HZ09}
W. Hebisch and B. Zegarlinski.
\newblock Coercive inequalities on metric measure spaces.
\newblock {\em J. Funct. Anal.}, 258:814--851, 2010.

\bibitem[HN79]{HN79}
B.~Helffer and J.~Nourrigat.
\newblock Caracterisation des op\'erateurs hypoelliptiques homog\`enes
  invariants \`a gauche sur un groupe de {L}ie nilpotent gradu\'e.
\newblock {\em Comm. Partial Differential Equations}, 4(8):899--958, 1979.




\bibitem[KRS20]{KRS20}
A. Kassymov, M. Ruzhansky and D. Suragan. \newblock Fractional logarithmic inequalities and blow-up results with logarithmic nonlinearity on homogeneous groups.
\newblock {\em Nonlinear Differ. Equ. Appl.} 27:7, 2020.

\bibitem[KS20]{KS20}
A. Kassymov and D. Suragan. \newblock Fractional Hardy–Sobolev Inequalities and Existence Results for Fractional Sub-Laplacians.
\newblock {\em Journal of Mathematical Sciences,} 250(2):337--350, 2020.

\bibitem[Rel56]{Rel56}
F. Rellich.
\newblock Halbbeschränkte Diﬀerentialoperatoren höherer Ordnung, in: Proceedings of the International Congress of Mathematicians, vol. III, Amsterdam, 1954, Erven P. Noordhoﬀ N.V./North-Holland Publishing Co., Groningen/Amsterdam, 1956, pp. 243–250.











\bibitem[Ros76]{Ros76}
J. Rosen.
\newblock Sobolev inequalities for weight spaces and supercontractivity.
\newblock {\em Trans. Amer. Math. Soc.}, 222:367--376, 1976.



\bibitem[RTY20]{RTY20}
M. Ruzhansky, N. Tokmagambetov and N. Yessirkegenov.
\newblock Best constants in Sobolev and Gagliardo-Nirenberg inequalities on graded groups and ground states for higher order nonlinear subelliptic equations.
\newblock {\em Calc. Var. Partial Differential Equations}, \textbf{59}, no. 175, 23pp, 2020.

\bibitem[RY18a]{RY18a}
M. Ruzhansky and N. Yessirkegenov. Hardy, Hardy-Sobolev, Hardy-Littlewood-Sobolev and Caffarelli-Kohn-Nirenberg inequalities on general Lie groups. {\em  arXiv:1810.08845}, 2018.

\bibitem[RY18b]{RY18b}
M. Ruzhansky and N. Yessirkegenov. Hypoelliptic functional inequalities. {\em  arXiv:1805.01064}, 2018.



\bibitem[RS19]{RS19}
M. Ruzhansky and D. Suragan.
\newblock \textit{ Hardy inequalities on homogeneous groups: 100 years of Hardy inequalities}.
\newblock Progress in Math., Vol. 327, Birkh\"auser/Springer, Cham, 2019. xvi+571pp.



\bibitem[Tos97]{Tos97}
G. Toscani.
\newblock Sur l' in\'egalit\'e logarithmique de Sobolev.
\newblock {\em C.R. Acad. Paris}, 324(1):689--694, 1997.



\bibitem[SZ92]{SZ92}
D. W. Stroock and B. Zegarlinski.
\newblock The logarithmic Sobolev inequality for continuous spin systems on a lattice.
\newblock {\em J. Funct. Anal.}, 104(2):299--326, 1992.


\bibitem[VSCC93]{VSCC93} N. Th. Varopoulos,
L. Saloff-Coste and T. Coulhon.
\newblock {\em Analysis and geometry on groups}.
\newblock Cambridge: Cambridge University Press, 1993.

\bibitem[Wei78]{Wei78}
F. B. Weissler.
\newblock Logarithmic Sobolev inequalities for the heat-diffusion semigroup.
\newblock {\em Trans. Amer. Math. Soc.}, 237:255--269, 1978.

\end{thebibliography}
\end{document}